\documentclass[twoside, centertags]{amsart}

\usepackage[cmtip,color,all]{xy}
\usepackage{amsmath}
\usepackage{amsfonts}
\usepackage{amssymb}
\usepackage{amsthm}
\usepackage{graphicx}
\usepackage{mathrsfs}
\usepackage{hyperref}
\usepackage{xcolor}
\usepackage{parskip}

\DeclareMathAlphabet{\matheurm}{U}{eur}{m}{n}

\newtheorem{theorem}{Theorem}[section]
\newtheorem{corollary}[theorem]{Corollary}

\newtheorem{proposition}[theorem]{Proposition}
\newtheorem{conjecture}[theorem]{Conjecture}

\theoremstyle{definition}
\newtheorem{definition}[theorem]{Definition}
\newtheorem{remark}[theorem]{Remark}
\newtheorem{example}[theorem]{Example}
\newtheorem{construction}[theorem]{Construction}

\DeclareMathAlphabet{\matheurm}{U}{eur}{m}{n}

\DeclareMathOperator*{\hocolim}{hocolim}

\DeclareMathOperator{\id}{id}

\newcommand{\cat}{\matheurm{Cat}}

\newcommand{\fibr}[2]{\begin{pmatrix} {#1} \\ \downarrow \\ {#2}\end{pmatrix}}

\newcommand{\RR}{\mathbb{R}}

\newcommand{\calC}{\mathrm{Cob}^\delta}

\newcommand{\dell}{\partial}

\newcommand{\op}{\mathrm{op}}
\newcommand{\E}{\matheurm{E}}

\newcommand{\C}{\mathcal{C}}

\renewcommand{\S}{\mathcal{S}}
\newcommand{\D}{\mathcal{D}}

\newcommand*{\defeq}{\mathrel{\vcenter{\baselineskip0.5ex \lineskiplimit0pt
                     \hbox{\scriptsize.}\hbox{\scriptsize.}}}%
                     =}
\newcommand{\Biv}{\matheurm{Biv}}
\newcommand{\Bun}{\matheurm{Bun}}
\newcommand{\BCob}{B\mathrm{Cob}}
\newcommand{\BCobSp}{\mathbf{BCob}}

\begin{document}

\title[Topological manifold bundles and the $A$-theory assembly map]{Topological manifold bundles \\and the $A$-theory assembly map} %
\author[G. Raptis]{George Raptis}
\author[W. Steimle]{Wolfgang Steimle}

\begin{abstract}
We give a new proof of an index theorem for fiber bundles of compact \emph{topological} manifolds due to Dwyer, 
Weiss, and Williams, which asserts that the parametrized $A$-theory characteristic of such a fiber bundle factors canonically
through the assembly map of $A$-theory. Furthermore our main result shows a refinement of this statement by providing 
such a factorization for an extended $A$-theory characteristic, defined on the parametrized topological cobordism 
category. The proof uses a convenient framework for bivariant theories and recent results of Gomez-Lopez and Kupers 
on the homotopy type of the topological cobordism category. We conjecture that this lift of the extended 
$A$-theory characteristic becomes highly connected as the manifold dimension increases.
\end{abstract}

\address{\newline
G. Raptis \newline
Fakult\"at f\"ur Mathematik, Universit\"at Regensburg, 93040 Regensburg, Germany}
\email{georgios.raptis@ur.de}

\address{\newline
W. Steimle \newline
Institut f\"ur Mathematik, Universit\"at Augsburg, 86135 Augsburg, Germany}
\email{wolfgang.steimle@math.uni-augsburg.de}

\maketitle

\section{Introduction}

In \cite{DWW}, Dwyer, Weiss, and Williams defined the \emph{parametrized $A$-theory characteristic} of a fibration $p \colon E\to B$ with homotopy finite fibers, a fundamental 
$K$-theoretic invariant of $p$ which generalizes the classical Euler characteristic. This invariant is a section $\chi(p)$ of the fibration $A_B(E)\to B$ that is obtained from $p$ 
by applying Waldhausen's $A$-theory functor fiberwise. The Index Theorem of \cite{DWW} for topological manifold bundles asserts that if the fibration $p$ is equivalent to a 
fiber bundle of compact topological manifolds, then $\chi(p)$ factors canonically through the fiberwise assembly map $A^\%_B(E)\to A_B(E)$. This theorem 
is an analogue of the smooth index theorem of \cite{DWW} and may be seen as a strong version of the Bismut--Lott index theorem \cite[Theorem 0.1]{BL} 
in the setting of topological manifold bundles.

In this paper we give a new proof of a strong version of this fundamental result. This builds on and improves ideas of our approach to the corresponding index theorem of \cite{DWW} in the case of smooth manifolds (see \cite{RS1, RS2}). 
A main ingredient of the proof is the result on the homotopy type of the topological cobordism category with tangential structure, which was recently obtained by Gomez-Lopez and Kupers \cite{GLK}. This homotopy type has a formally similar description as in the smooth case, but it is much less tractable, due to the appearance of topological Grassmannians, so the arguments from \cite{RS1, RS2} do not apply exactly in this context. However, we show that the precise identification of this homotopy type is not required, but what matters is the fact that it is \emph{excisive} in  the tangential structure. This analysis is carried out in the formalism of \emph{bivariant theories}, in which we state and prove a general bivariant index theorem.

The connection between the Index Theorem and cobordism categories is based on the fact that the parametrized $A$-theory characteristic can naturally be extended to a map on the classifying 
space of the cobordism category. This was observed in \cite{BM, RS1} for the smooth cobordism category, and in \cite{RS2} the authors improved this to a \emph{bivariant} 
transformation from a bivariant version of the cobordism category to bivariant $A$-theory. General bivariant theories come with universal constructions, such as coassembly and assembly 
transformations, whose study is related to index type theorems. This leads to Theorem \ref{formal_index_theorem} which is a formal version of the topological Dwyer--Weiss--Williams index theorem in the abstract setting of bivariant theories. 
Theorem \ref{main_index_thm} specializes this general result to a bivariant index theorem for topological manifold bundles, from which we deduce 
the Index Theorem of \cite{DWW} (Corollary \ref{DWW_thm}). 

The proof given here applies also in the smooth category with only minor modifications. Therefore, the parallel formulations of the Index Theorem in the categories of topological and smooth manifolds are now paired with parallel methods of proof. However, unlike in the case of topological manifolds, the proof in the smooth setting in \cite{RS1, RS2} required the detailed identification of the homotopy type of the cobordism category, in order to identify the map 
from the cobordism category to $A$-theory. In this connection, we conjecture an analogous identification of this map in the case of the topological cobordism category (Conjecture \ref{conjecture}).

\medskip

\noindent \textbf{Acknowledgements.} The first--named author thanks the Mathematical Institute, University of Oxford, for the support and the hospitality during an academic visit while 
this work was in preparation. He was also supported by the \emph{SFB 1085 -- Higher Invariants} (University of Regensburg) funded by the DFG. The second--named author was partially supported by the \emph{SPP 2026 -- Geometry at infinity} funded by the DFG. 

\section{Bivariant Theories}

\subsection{Preliminaries} A notion of bivariant theory was introduced in \cite{RS2} in order to fomalize the 
functoriality properties of the parametrized cobordism category of  compact smooth manifolds. We consider here a small modification of this notion which is better suited for the corresponding parametrized cobordism category of compact topological manifolds that will be defined in the next section.

We fix an integer $d \geq 0$. A \emph{family of $\RR^d$-bundles} is a triple 
\[\theta = (B,\; p\colon X\to B, \;\xi\colon V\to X)\]
where $B$ is a space which has the homotopy type of a CW complex, $p$ is a fibration, 
and $\xi$ is a numerable topological $\RR^d$-bundle. We additionally assume that $X$ is a subset of $B\times \mathcal U$ 
and that  $V$ is a subset of $X \times \mathcal U$, for a fixed set $\mathcal U$ of sufficiently high cardinality, in such 
a way that the respective maps to $B$ and $X$ are given by the projection. Each such triple gives rise to a notion of tangential structure for 
$B$-parametrized families of topological $d$-manifolds. Given two families of $\RR^d$-bundles $\theta = (B,p,\xi)$ and 
$\theta = (B, p', \xi')$ with the same base space $B$, a \emph{bundle map} $b \colon \theta \to \theta'$ consists of a fiberwise map $p\to p'$ which is covered by 
a fiberwise homeomorphism $\xi\to \xi'$.

Given a family of $\RR^d$-bundles $\theta = (B, p, \xi)$ and a map $g \colon B' \to B$, where $B'$ also has the homotopy type of
a CW complex, there is a new family of $\RR^d$-bundles: 
\[g^* \theta \defeq (B',\; p'\colon g^* X\to B, \;\xi'\colon g^*V\to g^*X)\]
where $g^*X$ and $g^*V$ are the pull-backs of $X$ and $V$ along $g$ (viewed as subsets of $B'\times \mathcal U$ and of 
$g^*X\times \mathcal U$, respectively), and $p'$ and $\xi'$ are the canonical maps induced by $p$ and $\xi$. A bundle map 
$b\colon \theta\to \theta'$ induces functorially (in $b$) a bundle map $g^*b\colon g^*\theta\to g^*\theta'$. Note also that 
the rule $g \mapsto g^*$ is itself functorial, in the sense that we have $g^*f^*\theta=(f\circ g)^*\theta$ and $\id^*\theta=\theta$.

Families of $\RR^d$-bundles are the objects of a category $\Biv$, where a morphism $(B, p, \xi) \to (B', p', \xi')$ consists of a map $g\colon B'\to B$ together with a bundle map $b\colon g^*(B, p, \xi)\to (B', p', \xi')$; the composition is defined by the rule
\[(h, c)\circ (g,b)\defeq (g\circ h, c\circ h^* b).\]

\begin{definition}
A \emph{bivariant theory} with values in a category $\E$ is a functor 
$$\C \colon \Biv\to \E.$$
\end{definition}

Explicitly, a bivariant theory $\C$ consists of the following assignments:
\begin{itemize}
\item[(a)] for each family of $\RR^d$-bundles $\theta=(B, p, \xi)$, an object $\C(\theta)$ of $\E$;
\item[(b)] for each family of $\RR^d$-bundles $\theta$ and each map $g \colon B' \to B$, a morphism (\emph{contravariant operation}) $g^* \colon \C(\theta) \to \C(g^*\theta)$ in $\E$; 
\item[(c)] for each bundle map of families of $\RR^d$-bundles $b\colon \theta\to \theta'$, a morphism (\emph{covariant operation}) $b_* \colon \C(\theta) \to \C(\theta')$ in $\E$,
\end{itemize}
such that 
\begin{enumerate}
 \item the collection of the morphisms $g^*$ satisfies the standard properties for contravariant functoriality;
 \item the collection of the morphisms $b_*$ satisfies the standard properties for covariant functoriality;
 \item the covariant and contravariant operations commute with each other in the sense that each of the following squares is commutative:
$$
\xymatrix{
\C(\theta) \ar[r]^{g^*} \ar[d]^{b_*} & \C(g^*\theta) \ar[d]^{(g^*b)_*} \\
\C(\theta') \ar[r]^{g^*} & \C(g^*\theta'). 
}
$$   
\end{enumerate}

\begin{remark}
This definition of bivariant theory is similar to the one considered in \cite{RS2}. 
The main difference is that we now also allow bundles $\xi$ which are not vector 
bundles and the datum $\xi$ is not presented in terms of a classifying map to $\mathrm{BO(d)}$ or $\mathrm{BTop(d)}$. 
This notion of bivariant theory is also closely related to the definition of bivariant theory due 
to Fulton--MacPherson \cite{FM} with the difference that we do not require the structure of product operations.
\end{remark}

A \emph{bivariant transformation} $\tau \colon \C \to \D$ is defined to be a natural transformation of functors. Explicitly, 
this consists of a collection of morphisms in $\E$,  $\tau(\theta) \colon \C(\theta) \to \D(\theta)$, for each family of 
$\RR^d$-bundles $\theta$, which is natural with respect to the covariant and contravariant operations. 
Suppose now that $\E$ is a category with weak equivalences (for example, the category of spaces or spectra with the standard 
classes of weak equivalences). A bivariant transformation is a \emph{weak equivalence} if it is given by weak equivalences 
for each family of $\RR^d$-bundles. 

We call a bivariant theory \emph{homotopy invariant} if the contravariant operation $g^*$ is a weak equivalence in $\E$ when $g$ is a 
homotopy equivalence (equivalently, weak homotopy equivalence) and the contravariant operation $b_*$ is a weak equivalence in $\E$ when $b$ is a weak homotopy 
equivalence. (We call a bundle map $b\colon (B, p \colon X \to B, \xi) \to (B, p'\colon X'\to B, \xi')$ a weak homotopy equivalence 
if the underlying map $X \to X'$ is a weak homotopy equivalence.)

\begin{example}(Bivariant A-theory) \label{bivariant_A_theory} The assignment
\[(B, p, \xi) \mapsto \mathbf A(p)\]
(bivariant $A$-theory spectrum of the fibration $p$) extends canonically to a bivariant theory; see \cite{Wi} or \cite[Section 3]{RS1}. 
This theory is homotopy invariant (see the proof of \cite[Proposition 3.6]{RS1} for covariant 
homotopy invariance and \cite[Proposition 3.8]{RS1} for contravariant homotopy invariance; we note that the first--mentioned 
proof also applies to the class of weak homotopy equivalences). Note that the $\RR^d$-bundle $\xi$ plays no role in 
the definition of $\mathbf A$.
\end{example}

A bivariant theory gives rise to a collection of covariant and contravariant functors. 
We will be interested in the following two types of functors that arise from a bivariant theory. 
Let $\Bun$ denote the category of numerable topological $\RR^d$-bundles. More precisely, $\Bun$ 
is the full subcategory of $\Biv$ on objects $(B, p, \xi)$ where $B=*$, that is, the objects of $\Bun$ 
are numerable topological $\RR^d$-bundles $(\xi \colon V \to X)$ (where $X$ is a subset of $\mathcal U$ 
and $V$ is a subset of $X\times \mathcal U$), and a morphism is a map of base spaces covered by a fiberwise 
homeomorphism.

\begin{definition}
The \emph{covariant part} $\overline{\C} \colon \Bun \to \E$ of a bivariant theory $\C$ is the restriction of $\C$ to $\Bun$.
\end{definition}

As the following construction shows, there is also a reverse process that takes covariant functors on $\Bun$ to bivariant 
theories. 

\begin{construction}(Associated bivariant theory) \label{assoc-biv-theory}
Let $F \colon \Bun \to \E$ be a functor where $\E$ is the category of spaces or spectra. Suppose that $F$ is homotopy invariant, 
i.e.,  it sends weak homotopy equivalences to weak equivalences in $\E$. Following the construction of \cite[Subsection 4.2]{RS2}, there 
is an associated bivariant theory $F^{\&}$ such that $\overline{F^{\&}}$ is $F$ (up to canonical weak equivalence). The assignment
$F \mapsto F^{\&}$ is functorial in $F$. More specifically, the value of $F^{\&}$ at $(B, p \colon X \to B, \xi \colon V \to X)$ 
is given by the space (or spectrum) of sections of the fibration $F_B(\xi) \to B$ whose fiber at $x \in  B$ is given by 
$F(\xi_{|p^{-1}(x)} \colon V_{| p^{-1}(x)} \to p^{-1}(x))$. In other words, $F^{\&}(B, p, \xi)$ is the homotopy limit 
of $F$ restricted to the fibers of $p$, where these fibers are regarded as defining a classifying diagram for $p$ with values 
in $\Bun$.   
\end{construction}

On the other hand, a bivariant theory $\C$ restricts to a collection of contravariant functors as follows. 
Let $\S_B$ denote the category of spaces over $B$ which are of the homotopy type of a CW complex. For any family of 
$\RR^d$-bundles $\theta = (B, p, \xi)$, we can view $\S_B$ as a subcategory of $\Biv$ by sending $(g\colon B'\to B)$ 
to the family of $\RR^d$-bundles $g^*\theta$. 
As a consequence, a bivariant theory $\C$ restricts to a (contravariant) functor:
$$\C_{/\theta} \colon \S_B^{\op} \to \E, \ \ (g \colon B' \to B) \mapsto \C(g^*\theta).$$

\subsection{Coassembly} The construction of the coassembly transformation for bivariant theories was introduced in \cite[Subsection 4.2]{RS2}. We recall here some facts about this construction.
\emph{We restrict throughout to bivariant theories with values in the category $\E$ of spaces or spectra, equipped with the usual class of weak equivalences.}

A homotopy invariant bivariant theory $\C$ is \emph{contravariantly excisive} if the functor $\C_{/\theta}$ is excisive for every $\theta=(B,p,\xi)$ (that is, if $\C_{/\theta}$ sends homotopy colimits to homotopy limits) -- this was called \emph{strongly excisive} in \cite{RS2}. This property essentially says that $\C$ 
is cohomological in $B$ with respect to the contravariant functoriality. If $\C_{/\theta}$ is an excisive functor with values in spectra, then it gives rise to a cohomology theory on spaces over $B$, with $B$-twisted coefficients given by the values of $\C_{/\theta}$ at $(\{x\}, p^{-1}(x) \to \{x\}, \xi_{|p^{-1}(x)})$ for each $x \in B$. 

\begin{example}
Let $F \colon \Bun \to \E$ be a functor where $\E$ is the category of spaces or spectra. Then the associated bivariant theory $F^{\&}$ of Example \ref{assoc-biv-theory} is contravariantly excisive by construction.  
\end{example}

A bivariant transformation $\tau \colon \C \to \D$ of homotopy invariant bivariant theories is a \emph{bivariant coassembly map} if $\D$ is contravariantly excisive and $\tau$ restricts to a weak equivalence of covariant functors $\overline\tau \colon \overline\C \to \overline\D$. If $\C$ is contravariantly excisive, then any bivariant coassembly map $\C \to \D$ is necessarily a weak equivalence of bivariant theories. A bivariant coassembly map for bivariant theories was constructed in \cite[Subsection 4.2]{RS2}
-- the construction applies similarly to our present context. Given a bivariant theory $\C$, the bivariant coassembly map for $\C$ is given by a canonical bivariant transformation:
$$\nabla_{\C}\colon \C \to \C^{\&}\defeq (\overline \C)^\&$$
that is defined essentially by the canonical maps to the respective homotopy limits. 

We summarize the properties of the bivariant coassembly map in the next proposition. We write $[-,-]$ to denote the morphism sets in the homotopy categories of the functor categories $\E^{\Biv}$ and $\E^{\Bun}$, respectively, obtained by formally inverting the (pointwise) weak equivalences.

\begin{proposition}\label{maps_into_excisive_theory}
Let  $\C$ and $\D$ be homotopy invariant bivariant theories and suppose that $\D$ is contravariantly excisive. Then the functor $\C \mapsto \overline \C$ induces a bijection of morphism sets:
\[[\C,\D] \xrightarrow{\cong} [ \overline\C, \overline \D].\] 
As a consequence, the bivariant coassembly map $\nabla_{\C}$ induces a bijection of morphism sets in the homotopy category of bivariant
theories: 
$$
\nabla_{\C}^*: [\C^\&, \D] \xrightarrow{\cong} [\C, \D].
$$
\end{proposition}
\begin{proof}
The proof is similar to \cite[Proposition 4.3]{RS2}.
\end{proof}

\subsection{A formal index theorem}

We can similarly consider bivariant theories $\C$ whose covariant part $\overline \C$ satisfies excision. By definition, 
a functor $F \colon \Bun \to \E$, where $\E$ is the category of spaces or spectra, is excisive if it is homotopy 
invariant and it preserves homotopy colimits. (It may be helpful here to identify $\Bun$ with the category of spaces over 
$\mathrm{BTop(d)}$, as homotopy theories.) 

\begin{definition} [Fully excisive]
Let $\C$ be a homotopy invariant bivariant theory with values in the category of spaces or spectra. We say that 
$\C$ is \emph{fully excisive} if it is contravariantly excisive and $\overline \C$ is excisive.
\end{definition}

\begin{construction}[Assembly] \label{assembly}
Let $F \colon \Bun \to \E$ be a homotopy invariant functor where $\E$ is the category of spaces or spectra. Following the construction of assembly in \cite{WW}, there is an excisive functor $F^{\%} \colon \Bun \to \E$ and a natural transformation $\alpha_F \colon F^{\%} \to F$ which defines a universal approximation by an excisive functor: for any excisive functor $H \colon \Bun \to \E$, there is a bijection of morphism sets:
\begin{equation} \label{assembly-bij}
\alpha_F \circ - \colon [H, F^{\%}] \xrightarrow{\cong} [H, F].
\end{equation} 
The associated bivariant theory $(F^{\%})^{\&}$ is a fully excisive bivariant theory.  
\end{construction}

\begin{theorem} \label{formal_index_theorem}
Let $\tau \colon \C \to \D$ be a bivariant transformation between homotopy invariant bivariant theories with values 
in the category of spaces or spectra. Suppose that $\overline \C$ is 
excisive. Then there is a unique bivariant transformation in the homotopy category of bivariant theories,
$$\tau^{\%} \colon \C^{\&} \longrightarrow (\overline \D^{\%})^{\&},$$
such that the following diagram commutes in the homotopy category of bivariant theories:
$$
\xymatrix{
\C \ar[d]_{\nabla_{\C}} \ar[rr]^{\tau} && \D \ar[d]^{\nabla_{\D}} \\
\C^{\&} \ar[r]^(0.4){\tau^{\%}} & (\overline  \D^{\%})^{\&} \ar[r]^{\alpha_{\overline \D}^{\&}} & \D^{\&}. 
}
$$
\end{theorem}
\begin{proof}
By the naturality of the coassembly transformation, we obtain a commutative diagram as follows, 
$$
\xymatrix{
\C \ar[d]_{\nabla_{\C}} \ar[r]^{\tau} & \D \ar[d]^{\nabla_{\D}} \\
\C^{\&} \ar[r]_{\tau^{\&}} & \D^{\&}. 
}
$$
The bottom transformation $\tau^{\&}$ factors uniquely through the canonical bivariant transformation $\alpha_{\overline \D}^{\&}$ using the bijections of
Proposition \ref{maps_into_excisive_theory} and Construction \ref{assembly}.
\end{proof}

\begin{remark}
We may view Theorem \ref{formal_index_theorem} as an abstract index type theorem in the following way. Each class $x \in  \pi_0 \C(\theta)$ gives rise to 
two characteristic cohomology classes $(\tau^{\%} \circ \nabla_{\C})(x)$ and $(\nabla_{\D} \circ \tau)(x)$. The theorem then indentifies these two 
classes along the canonical assembly transformation $\alpha^{\&}_{\overline \D}$. In certain special cases of $\C$ and $x$, these two chararacteristic 
classes are related to transfer constructions. 
We refer to \cite{Wi} for a nice overview of this idea.   
\end{remark}

\begin{remark} \label{stronger_univ_coassembly}
By passing to the homotopy category, we contend ourselves here with proving less than what is possible. 
Since all of our constructions are homotopy coherent and our identifications are canonical, 
Theorem \ref{formal_index_theorem} can also be formulated in the homotopy theory of bivariant theories. For this purpose, 
it would be more convenient to consider coassembly and assembly as parts of adjunctions between the respective $\infty$-categories 
of functors.
\end{remark}

\section{An Index Theorem for Topological Manifold Bundles}

\subsection{The (parametrized) topological cobordism category} The topological cobordism category was introduced and studied in \cite{GLK}. Following the definition of the parametrized \emph{smooth} cobordism category in \cite{RS2}, we also define a parametrized bivariant extension of the topological cobordism category.

Let $\theta = (B, p \colon X \to B, \xi \colon V \to X)$ be a family of $\RR^d$-bundles. There is a (discrete) category $\calC(\theta)$ 
of parametrized topological $\theta$-cobordisms over $B$ defined as follows. An object in $\calC(\theta)$ is given by a quadruple $(E, \pi, a, l)$ where:
\begin{itemize}
\item[(i)] $a\in\RR$,
\item[(ii)] $\pi \colon E\to B$ is a numerable fiber bundle of compact $(d-1)$-dimensional topological manifolds, which is fiberwise embedded in $B\times \{a\}\times \RR_+ \times \RR^{\infty}$ 
and the embedding is cylindrical near the fiberwise boundary $\partial_{\pi} E \subset E$, 
\item[(iii)] $l$ is a tangential $\theta$-structure, i.e., a microbundle map $\epsilon \oplus T_{\pi} E \to \xi$ (fiberwise over $B$ and cylindrical near the fiberwise boundary), where $T_{\pi} E$ denotes the vertical tangent microbundle of $\pi$,
and $\epsilon$ is the trivial $\RR$-bundle.
\end{itemize}

A morphism in $\calC(\theta)$ consists of $a_0<a_1\in \RR$ and a numerable fiber bundle of compact topological $d$-manifolds,
\begin{equation} \label{def_morphism}
\pi\colon W\to B,
\end{equation} 
embedded fiberwise in $B\times [a_0,a_1]\times \RR_+ \times \RR^{\infty}$ and cylindrically near the boundary and the corners, together with 
a tangential $\theta$-structure, given by a microbundle map $l_W \colon T_{\pi} W\to \xi$, fiberwise over $B$ and cylindrical near the fiberwise boundary parts. The domain and target of this morphism are 
the intersections $W_0$ and $W_1$ of $W$ with $B \times \{a_0\}\times \RR_+ \times \RR^{\infty}$ and 
$B \times \{a_1\}\times \RR_+ \times \RR^{\infty}$ respectively, together with the restrictions of $l_W$ to these subsets. 
This defines a non-unital category where the composition of morphisms is given by union of subsets in 
$B \times \RR \times \RR_+ \times \RR^{ \infty}$. 

\begin{remark}
For bundles $\pi\colon W \to B$ of compact topological manifolds with \emph{boundary}, as in the previous definition, we find it convenient to define the 
vertical tangent microbundle $T_\pi W$ as the vertical tangent microbundle of the fiberwise \emph{horizontal interiors}, that is, 
the intersection of $W$ with $B\times [a_0, a_1]\times (0, \infty)\times \RR^{\infty}$. (Since the inclusion of the fiberwise 
horizontal interior into the whole bundle is a fiberwise homotopy equivalence, this does not conflict with other possible 
definitions.)
\end{remark}

\begin{remark}
Since every numerable fiber bundle is a fibration, the fiber bundles in (ii) and \eqref{def_morphism} are also fibrations. In addition, since the fibers of 
these fibrations have the homotopy type of a CW complex, it follows that the same holds for the total spaces (see the proof of \cite[Lemma A.1]{RS1}). 
\end{remark}

Given a map $g \colon B' \to B$, we get an induced functor (\emph{contravariant operation})
\[g^*\colon \calC(\theta)\to\calC(g^*\theta)\]
which is defined by taking pullbacks of bundles along $g$. On the other hand, if $\theta = (B, p \colon X \to B, \xi)$ and $\theta' = (B, q \colon Y \to B, \xi')$ are families 
of $\RR^d$-bundles and $b \colon \theta \to \theta'$  is a bundle map, then post-composing with $b$ defines a functor 
(\emph{covariant operation})
\[b_*\colon \calC(\theta) \to \calC(\theta').\]
The operations of $b_*$ and $g^*$ are clearly functorial and commute with each other. 
Thus, the assignment $\calC \colon \theta\mapsto \calC(\theta)$ is a bivariant theory with values in the category of small non-unital categories $\cat$.

\noindent \textbf{Remark on notation.} We will only consider cobordism categories of \emph{topological} manifolds and we will always allow the objects to have \emph{boundary}. Thus, in order to simplify the notation, we will use throughout the notation $\calC(-)$ without any of the decorations that usually 
indicate these choices.

Following \cite[Section 2]{RS2}, we also consider the associated simplicial thickening of this bivariant theory $\calC(\theta)_\bullet$. We recall that for a family of $\RR^d$-bundles $\theta = (B, p \colon X \to B, \xi \colon V \to X)$, $\calC(\theta)_\bullet$ is  
a simplicial category (= simplicial object in $\cat$) which is defined degreewise by 
\[\calC(\theta)_n:=\calC(\theta\times \id_{\Delta^n}),\]
where 
\[\theta\times\id_{\Delta^n} = (B \times \Delta^n, p \times \id_{\Delta^n} X \times \Delta^n \to B\times \Delta^n, \xi \times \id_{\Delta^n} \colon V \times \Delta^n \to X \times \Delta^n).\]
The simplicial operators are defined by the contravariant operations of the bivariant theory $\calC(-)$. 

For every small non-unital category $C$, the \emph{nerve} $N_{\bullet} C$ of $C$ is a semi-simplicial set, and the \emph{classifying space} of $C$, denoted by $BC$, is the geometric realization of the nerve $N_{\bullet} C$. 

\begin{definition} \label{def_cob_cat}
The (fat) geometric realization of the degreewise classifying spaces 
\[\BCob(\theta):= \vert B\calC(\theta)_\bullet\vert\]
is the \emph{classifying space of the parametrized topological $\theta$-cobordism category.}
\end{definition}
 
By construction, the rule $\theta\mapsto \BCob(\theta)$ canonically extends to a bivariant theory with values in spaces.  Definition \ref{def_cob_cat} is in accordance with the existing definition of the topological cobordism category. Indeed, the covariant 
part of $\BCob$ is equivalent to the functor $\xi \mapsto B\mathsf{Cob}^{\operatorname{Top}, \xi}_{\partial}(d, \infty)$ from \cite[section 7.4]{GLK}. 

\begin{proposition} \label{prop_homotopy}
The bivariant theory $\BCob$ is homotopy invariant.
\end{proposition}
\begin{proof}
The proof in \cite[Proposition 2.4]{RS2} shows contravariant homotopy invariance and covariant homotopy invariance with respect to the homotopy equivalences. 
We show that $\BCob$ is covariantly homotopy invariant also with respect to the weak homotopy equivalences. 
Let $\theta = (B, p \colon X \to B, \xi)$ be a family of $\RR^d$-bundles. Let $g \colon X' \to X$ be a map which is a fibration and a weak homotopy equivalence and where $X'$ has the homotopy type of a CW complex. Let $\theta' = (B, p \circ g \colon X' \to B, \xi')$ be the new family of $\RR^d$-bundles, defined by pullback. Then it suffices to show that $\BCob$ sends the bundle map $b \colon \theta' \to \theta$ to a weak equivalence. For $k \geq 0$, consider the map of simplicial sets,
$$N_k \calC(\theta')_\bullet \to N_k \calC(\theta)_\bullet.$$
We claim that this map is a trivial Kan fibration, for each $k \geq 0$, from which the required result follows. The claim amounts to solving lifting problems over $B \times \Delta^n$ of the form:
$$
\xymatrix{
W_{|B \times \partial \Delta^n} \ar[r] \ar@{>->}[d]_i &  X'  \times \Delta^n \ar@{->>}[d]^{\sim} \\
W \ar[r] \ar@{-->}[ru] & X \times \Delta^n, 
}
$$
where $\pi \colon W \to B \times \Delta^n$ is an element in $N_k \calC(\theta)_n$ and $W_{| B \times \partial \Delta^n}$ denotes the restriction over 
$B \times \partial \Delta^n$. The existence of these lifts can be shown directly using standard arguments, or by appealing to the mixed model structure on the category of topological spaces 
\cite[Theorem 2.1, Example 2.2]{Co}, in which $g$ is a trivial fibration by definition, and $i$ is a cofibration by \cite[Corollaries 3.7 and 3.12]{Co}.  
\end{proof} 

The bivariant theory $\BCob$
lifts, through the functor $\Omega^\infty$, to a bivariant theory with values in the category of spectra, which we denote by $\BCobSp$. 
As in the smooth case \cite[Section 2]{RS2}, this can be shown by using the partial monoidal structure on $\calC(\theta)_{\bullet}$ given by union of subsets, 
whenever this is well-defined. This structure gives rise to a group-like (special) $\Gamma$-space which models an 
infinite loop space. The $\Gamma$-space structure can be described more precisely by varying the tangential $\theta$-structure 
as follows: the value of the $\Gamma$-space at the pointed set $n_+$ is 
$$\BCob(\theta(n_+))$$
where  
$$\theta(n_+) = (B, \; \coprod_n X \xrightarrow{ (p,\dots, p)} B, \;\coprod_n V \xrightarrow{\coprod_n \xi} \coprod_n X).$$
We omit the details as the arguments are similar to the smooth case (see \cite[Section 2]{RS2}, \cite{Ng}). 

\begin{proposition} \label{cob_excisive}
The covariant part of (the spectrum-valued theory) $\BCobSp$ is excisive.
\end{proposition}
\begin{proof}
As explained in \cite[Subsection 7.4]{GLK}, the results of \cite[Section 5]{GLK} generalize to manifolds with boundaries. In particular,  \cite[Corollary 5.8]{GLK} has an analogue for manifolds with boundary,
\begin{equation}\label{eq:cob_with_boundary}
\BCob(\xi) \xrightarrow{\simeq} \Omega^{\infty} \mathbf B(\xi),
\end{equation}
where $\mathbf B(\xi)$ denotes the suspension of a spectrum $\Psi^{\mathrm{Top}, \xi}_\dell(d)$, whose $n$-th term is 
$$\Psi^{\mathrm{Top}, \xi}_\dell(d)_n = \psi^{\mathrm{Top}, \xi}_\dell(d,n+1,n);$$
here we recall that $\psi^{\mathrm{Top}, \xi}_\dell(d,n,p)$ denotes the space of $d$-dimensional $\xi$-manifolds, possibly with boundary, neatly embedded in $\RR^p \times (0,1)^{n-p-1}\times [0,1)$.

The functor $\mathbf B (-)$ is invariant under weak equivalences and commutes with geometric realizations up to weak equivalence; this is shown for the 
non--boundary version $\Psi^{\mathrm{Top}, (-)}(d)$ in \cite[Lemma 7.3 and Theorem 7.4]{GLK} and it follows for the 
boundary--version $\Psi^{\mathrm{Top}, \xi}_\dell(d)$ (and therefore also for $\mathbf{B}(-)$) from the cofiber sequence of spectra explained in \cite[p.~50]{GLK}. 

Next we argue that $\Psi^{\mathrm{Top}, \xi}_\dell(d)$ preserves coproducts in the $\xi$-variable up to weak equivalence. We can see this by observing that the spectrum $\Psi^{\mathrm{Top}, \xi}_\dell(d)$ is equivalent to the spectrum whose $n$-th term is $\psi^{\mathrm{Top}, \xi}_\dell(d,\infty,n)$ (with similarly defined structure maps)
and then using the fact that the spaces $\psi^{\mathrm{Top}, \xi}_\dell(d,\infty,n)$ preserve coproducts in the $\xi$-variable, up to weak equivalence, by the argument of \cite{Ng}.
Therefore, $\mathbf{B}(-)$ preserves coproducts up to weak equivalence. Using the Bousfield--Kan formula for general homotopy colimits, we conclude that $\mathbf{B}(-)$ preserves arbitrary homotopy colimits.

By the naturality of \eqref{eq:cob_with_boundary} in $\xi$, the equivalence \eqref{eq:cob_with_boundary} extends to an equivalence of $\Gamma$-spaces and therefore it induces an equivalence between the associated spectra. Since $\mathbf B(\xi)$ also defines a $\Gamma$-object in spectra, it follows \cite[Proposition 5.2]{Ng} that the spectrum associated with the $\Gamma$-space $\Omega^\infty (\mathbf B(\xi))$ is the connective cover $\mathbf B(\xi)_{\geq 0}$
of $\mathbf B(\xi)$. So, the equivalence \eqref{eq:cob_with_boundary} extends to an equivalence of (connective) spectra
\[\BCobSp(\xi)\to \mathbf B(\xi)_{\geq 0}.\]
Thus, in order to conclude the proof, it is enough to show that $\mathbf B(\xi)_{\geq 0}$ is again excisive in $\xi$. Clearly, it preserves small coproducts up to weak equivalence; we are left to show that it also preserves homotopy pushouts. The functor $(-)_{\geq 0}$ does not preserve general homotopy pushouts; but it does preserve those for which $\pi_0$ of each of the spectra in the homotopy pushout diagram vanishes. In our case, 
\[
\pi_0 \BCob(\xi)
\]
is given by bordism classes of $\xi$-manifolds with boundaries. Since any such $\xi$-manifold with boundary is canonically null--bordant, it follows that the classifying space is indeed connected. 
\end{proof}

\subsection{The parametrized $A$-theory characteristic} \label{def_biv_map}

There is a bivariant transformation
\begin{equation} \label{tau_spaces} 
\tau(\theta)\colon \Omega \BCob(\theta) \to \Omega^\infty \mathbf A\fibr{X}{B}, \quad \theta = (B, p\colon X \to B, \xi \colon V \to X), 
 \end{equation}
from the loop space of the bivariant theory defined by the parametrized topological cobordism category with boundary 
to bivariant $A$-theory, which is defined as in the smooth case \cite[5.1]{RS2} -- the construction does not use smoothness and therefore applies to the topological cobordism category as well. 
Roughly speaking, this transformation is given by viewing a chain of composable cobordisms as a filtration of 
their composite. Moreover, in terms of the cobordism model for $A$-theory presented in \cite[Section 4]{RS3}, it may be understood as the inclusion of the $\theta$-cobordism category into a 
``homotopy cobordism category'', which is a category of cospans of fiberwise homotopy finite spaces over $B$ with a structure map to $p$. The covariant part of this transformation was first considered by B\"okstedt--Madsen \cite{BM}.

Using the $\Gamma$-space method, the bivariant transformation \eqref{tau_spaces} may be refined to a bivariant transformation of spectrum-valued theories which we write as
\[\tau(\theta)\colon \Omega \BCobSp(\theta)\to \mathbf A\fibr X B.\]

\begin{theorem}\label{main_index_thm}
There
is a unique bivariant transformation in the homotopy category of bivariant theories with values in the category of spectra,
$$\tau^{\%} \colon \Omega \BCobSp^{\&} \longrightarrow (\overline{\mathbf A}^{\%})^{\&},$$
such that the following diagram commutes in the homotopy category of bivariant theories:
\[
\xymatrix{
 \Omega \BCobSp \ar[rr]^{\tau} \ar[d]_{\nabla} && \mathbf A \ar[d]^{\nabla_{\mathbf A}} \\
 \Omega \BCobSp^{\&} \ar[r]^{\tau^{\%}} & (\overline{\mathbf A}^{\%})^{\&} \ar[r]^{\alpha_{\overline{\mathbf A}}^{\&}} & \mathbf A^\&.
}
\]
\end{theorem}

\begin{proof}
This is  a direct application of Theorem \ref{formal_index_theorem} using Proposition \ref{cob_excisive}. 
\end{proof}

Next we explain how Theorem \ref{main_index_thm} specializes to the Index Theorem for the $A$-theory characteristic of fiber bundles of compact topological manifolds from \cite{DWW}. 
Let $\pi \colon E \to B$ be a fiber bundle of compact topological $d$-manifolds where $B$ is a CW complex. We may choose a fiberwise embedding of $\pi$ into 
$B \times (0,1) \times \RR_+ \times \mathbb{R}^\infty$, which is cylindrical near the boundary. We denote by $\xi \colon T_{\pi} E \to E$ the vertical tangent topological
$\RR^d$-bundle (using the Kister--Mazur theorem if necessary -- see \cite[Appendix A]{GLK}). In this way, we obtain a family of $\RR^d$-bundles $\theta \defeq (B, \pi, \xi)$ and
we may regard $(E, \pi, 0<1, \id)$ as a morphism in $\calC(\theta)$ (from $\varnothing$ to $\varnothing$). This morphism defines also a class in $\pi_0 \big(\Omega \BCob(\theta) \big)$
which we will denote by $[\pi]$.

The associated class $\tau[\pi]\in \pi_0\mathbf{A}(\theta)$ is given by the retractive space $E \sqcup E$ over $E$ -- the \emph{bivariant $A$-theory characteristic} 
of $p$, see \cite[Section 4]{RS1}. The image of this element under the coassembly map $\nabla_{\mathbf{A}}$ is the \emph{parametrized $A$-theory characteristic} of $p$,
$$\chi(\pi) \colon B \to A_B(E),$$
where $A_B(E) \to B$ is the fibration that is obtained from $\pi$ by applying the (space-valued) $A$-theory functor fiberwise. On the other hand, the image of $[\pi]$ under 
the bivariant transformation $(\tau^{\%} \circ \nabla)$ yields a class $\chi^{\%}(\pi) \in \pi_0 \big((\overline{\mathbf {A}}^{\%})^{\&}(\theta)\big)$, the \emph{excisive $A$-theory characteristic} of $p$, 
$$\chi^{\%}(\pi) \colon B \to A^{\%}_B(E),$$
where $A^{\%}_B(E) \to B$ is the fibration obtained from $\pi$ by applying the functor $A^{\%}$ fiberwise. Then the commutativity of the diagram in Theorem \ref{main_index_thm} 
yields the following result due to Dwyer--Weiss--Williams \cite{DWW}. 

\begin{corollary} \label{DWW_thm}
Let $\pi \colon E \to B$ be a fiber bundle of compact topological $d$-manifolds where $B$ is a CW complex. Then the following diagram commutes up to homotopy:
$$
\xymatrix{
& A^{\%}_B(E) \ar[d]^{\alpha_{\overline{\mathbf A}, B}} \\
B \ar[r]_{\chi(\pi)} \ar[ur]^{\chi^{\%}(\pi)} & A_B(E)
}
$$
where $A^{\%}_B(E) \xrightarrow{\alpha_{\overline{\mathbf A}, B}} A_B(E)$ denotes the $A$-theory assembly map fiberwise over $B$.  
\end{corollary}

We expect that the map $\chi^\%(\pi)$ agrees with the corresponding excisive $A$-theory characteristic as defined in \cite[7.11]{DWW}, which used methods 
of controlled $A$-theory in order to model the $A$-theory assembly map. We also expect that this identification can be shown by considering bivariant versions 
of the constructions in \cite[Section 7]{DWW} and appealing to the uniqueness property of $\tau^{\%}$ in Theorem \ref{main_index_thm}.

Given a family of $\RR^d$-bundles $\theta = (B, p, \xi)$, there is an associated family of $\RR^{d+1}$-bundles $\theta \oplus \epsilon \defeq (B, p, \xi \oplus \epsilon)$. There is a stabilization map
\[\calC(\theta) \to \calC(\theta\oplus \epsilon)\]
which sends an object $(E, \pi, a, l)$ to the object $(E'\defeq E\times [0,1], \pi', a, l')$, where $\pi'$ is the composite $E'\xrightarrow{\mathrm{proj}} E \xrightarrow{p} B$, and $l'$ is the composite
\[l' \colon T_{\pi'} E' \to T_{\pi} E \oplus \epsilon \xrightarrow{l\oplus\id} \xi\oplus\epsilon\]
where the first map is the canonical bundle map over the projection $E'\to E$. The neat embedding of $E'$ is the product embedding of the embedding of $E$ and a neat embedding $[0,1]\to \RR_+\times \RR$, followed by a suitable homeomorphism $\RR_+\times \RR_+\to \RR_+\times \RR$ which straightens the corner (compare \cite[Appendix A]{RS4}). This, together with a similar rule for morphisms, defines 
a bivariant transformation
\begin{equation}\label{eq:stabilization_functor}
-\times [0,1]\colon \calC(\theta) \to \calC(\theta\oplus\epsilon). 
\end{equation}
After simplicial thickening, geometric realization, and $\Gamma$-space delooping, we obtain a diagram of spectrum-valued bivariant theories
\[ \xymatrix{
 \Omega\BCobSp(\theta) \ar[rr]^{-\times [0,1]} \ar[rd]_\tau 
   && \Omega\BCobSp(\theta\oplus\epsilon) \ar[ld]^\tau\\
 & \mathbf A(p)
}\]
which commutes in the homotopy category. Hence the bivariant transformation $\tau$ for dimension $d$ factors through the one for dimension $d+1$. 

Based on an analogy with the smooth case \cite[Theorem 5.2]{RS2}, the results of \cite{GLK}, and the construction of the assembly map of $A$-theory 
in terms of higher simple homotopy theory \cite{Wa} (which is again related to stabilized manifold theory \cite{WJR}), we expect that the following
holds: 

\begin{conjecture} \label{conjecture}
The connectivity of the covariant part of $\tau^{\%}$ increases to $\infty$ as $d$ increases to $\infty$. 
\end{conjecture}

\noindent Assuming this, a suitably stabilized version of the topological B\"okstedt--Madsen map,
\[\tau\colon \hocolim_n \Omega \BCobSp(\xi\oplus\epsilon^n) \to \mathbf A(X),\]
would yield a model for the assembly map of $A$-theory, for any numerable $\RR^d$-bundle $\xi$ over $X$.


\begin{thebibliography}{10}


\bibitem{BL}
J.-M.~Bismut, J.~Lott, \emph{Flat vector bundles, direct images and higher real analytic
torsion.} J. Amer. Math. Soc.~8 (1995), no. 2, 291--363.

\bibitem{BM}
M.~B{\"o}kstedt, I.~Madsen, \emph{The cobordism category and {W}aldhausen's {$K$}-theory.} An alpine expedition through algebraic topology, pp.~39--80, Contemp. Math. Vol. 617, Amer. Math. Soc., Providence, RI, 2014.

\bibitem{Co}
M.~Cole, \emph{Mixing model structures.} Topology and its Applications 153 (2006), no. 7, 1016--1032.

\bibitem{DWW}
W.~Dwyer, M.~Weiss, and B.~Williams, \emph{A parametrized index theorem for the algebraic K-theory Euler class.} Acta Math. 190 (2003), no.1, 1--104. 

\bibitem{FM}
W.~Fulton, R.~MacPherson, \emph{Categorical framework for the study of singular spaces.}
Mem. Amer. Math. Soc. 31 (1981), no.~243.


\bibitem{GLK}
M.~Gomez-Lopez, A.~Kupers, \emph{The homotopy type of the topological cobordism category.} Preprint (2018), arXiv: \url{https://arxiv.org/abs/1810.05277v2}


\bibitem{Ng} 
H.~K.~Nguyen, \emph{On the infinite loop space structure of the cobordism category.} Algebr. Geom. Topol. 17 (2017), no. 2, 1021--1040.

\bibitem{RS1} 
G.~Raptis, W.~Steimle, \emph{On the map of B\"okstedt-Madsen from the cobordism category to A-theory.} Algebr. Geom. Topol. 14 (2014), no. 1, 299--347.

\bibitem{RS2}
G.~Raptis, W.~Steimle, \emph{Parametrized cobordism categories and the Dwyer-Weiss-Williams index theorem.} J. Topol. 10 (2017), no. 3, 700--719. 

\bibitem{RS3}
G.~Raptis, W.~Steimle, \emph{A cobordism model for Waldhausen $K$-theory.} J. London Math. Soc. 99 (2019), no. 2, 516--534.

\bibitem{RS4}
G.~Raptis, W.~Steimle, \emph{On the $h$-cobordism category. I.} \emph{To appear in} Int. Math. Res. Not. IMRN, \url{https://doi.org/10.1093/imrn/rnz329}.


\bibitem{Wa}
F.~Waldhausen, \emph{Algebraic K-theory of spaces.} Algebraic and geometric topology (New Brunswick, N.J., 1983), 
pp. 318-419, Lecture Notes in Math. 1126, Springer, Berlin, 1985. 

\bibitem{WJR} 
F.~Waldhausen, B.~Jahren, J.~Rognes, \emph{Spaces of PL manifolds and categories of simple maps.} Annals of Mathematics Studies, 186. 
Princeton University Press, Princeton, NJ, 2013. 

\bibitem{WW}
M.~Weiss, B.~Williams, \emph{Assembly.} Novikov conjectures, index theorems and rigidity, Vol. 2 (Oberwolfach, 1993), pp. 332--352, London Math. Soc. Lecture Note Ser. 227, Cambridge Univ. Press, Cambridge, 1995. 

\bibitem{Wi}
B.~Williams, \emph{Bivariant Riemann Roch theorems.} Geometry and topology: Aarhus (1998), pp. 377--393, Contemp. Math. Vol. 258, Amer. Math. Soc., Providence, RI, 2000. 

\end{thebibliography}
\end{document}